\documentclass[reqno]{amsart}
\usepackage{hyperref}
\usepackage{mathrsfs}
\usepackage{amssymb}
\begin{document}
\title[quantitative uniqueness of higher order equations]
{Quantitative uniqueness of some higher order elliptic equations}
\author{Shanlin Huang, Ming Wang, Quan Zheng}

\address{Shanlin Huang \newline
Department of Mathematics\\
Huazhong University of Science and Technology\\
Wuhan 430074, China}
\email{shanlin\_huang@hust.edu.cn}

\address{Ming Wang \newline
Department of Mathematics\\
China University of Geosciences\\
Wuhan 430074, China}
\email{mwangcug@outlook.com}

\address{Quan Zheng \newline
Department of Mathematics\\
 Huazhong University of Science and Technology\\
 Wuhan 430074, China}
\email{qzheng@hust.edu.cn}
%\author[1]{Shanlin Huang}
%\affil[1]{Huazhong University of Science and Technology, Wuhan 430074, China}

\begin{abstract}
We study the quantitative unique continuation property of some higher order elliptic operators. In the case of $P=(-\Delta)^m$, where $m$ is a positive integer, we derive lower bounds of decay at infinity for any nontrivial solutions under some general assumptions. Furthermore, in dimension  2, we can obtain  essentially sharp lower bounds for some forth order elliptic operators, the sharpness is shown by constructing  a  Meshkov type example.
\end{abstract}

%\thanks{Submitted July 21, 2013. Published November 13, 2013.}
%\thanks{Supported by DIME, Universidad Nacional de Colombia,
%Medell\'in, grant 15401.}
%\subjclass[2000]{35Q53, 37K05}

\keywords{Quantitative uniqueness; higher order elliptic operators;
\hfill\break\indent  estimates of Carleman type}
\thanks{This research is supported by National Natural Science Foundation of China under grant numbers 11471129}

\maketitle
\numberwithin{equation}{section}
\newtheorem{theorem}{Theorem}[section]
\newtheorem{lemma}[theorem]{Lemma}
\newtheorem{proposition}{Proposition}[section]
\newtheorem{remark}[theorem]{Remark}
\newtheorem{definition}[theorem]{Definition}
\newtheorem{example}[theorem]{Examples}

\allowdisplaybreaks

\section{Introduction}
In this note, we are interested in the following quantitative unique continuation problem at infinity of some
higher order elliptic operators with constant coefficients. Assume $u$ satisfies
\begin{equation}\label{1.1}
P(D)u+Vu=0, \quad\text{in}~~ \mathbb{R}^{n},
\end{equation}
and
\begin{align}\label{1.2}
|V|\leq C, ~~|u|\leq C,~~ u(0)=1.
\end{align}
For large $R$, one can define
$$
M(R)=\inf_{|x_0|=R}\sup_{B(x_0,1)}|u(x)|
$$
 to measure the precise decay information at infinity of the solution, then a natural question is how small can $M(R)$ be ? We first briefly recall the second order case, where a related problem was originally studied by Landis in 1960's \cite{KL}. He conjectured that if \eqref{1.1} and \eqref{1.2} are satisfied for $P=\Delta$, and $u(x)\leq C \exp\{  -C|x|^{1+}\}$ for some constant, then u is identically zero. This conjecture was disproved by Meshkov \cite{M} who  constructed non-trivial bounded, complex-valued functions $u,V$ satisfying \eqref{1.1} and $u(x)\lesssim e^{-C|x|^{\frac43}}$. In 2005, Bourgain and Kenig \cite{BK} derived a quantitative version of Meshkov's result in their resolution of Anderson localization for the Bernoulli model. More precisely, they showed that if \eqref{1.1} and \eqref{1.2} are satisfied for $P=\Delta$, then $M(R)\gtrsim \exp\{-CR^{\frac43}\log{R}\}$. This lower bound is sharp in view of Meshkov's example. Later this result was extended to the following general case by Davey \cite{D}
\begin{equation}\nonumber
\begin{gathered}
-\Delta u+W\cdot\nabla u+Vu=\lambda u,\\
|V|\leq C\langle x\rangle^{-N},~~|W(x)|\leq C\langle x\rangle^{-P},~~ |u|\leq C,
\end{gathered}
\end{equation}
for some $P, N>0$, where $\langle x\rangle=\sqrt{1+|x|^2}$. See \cite{LW}  for generalizations to  more general second order elliptic equations.
%The author proved that if we denote by $\beta_c=\max\{2-2P, \frac{4-2N}{3}\}$, then
%$$
%\inf_{|x_0|=R}\|u\|_{L^2(B(x_0, 1))}\gtrsim \exp\{-CR^{\beta_c}(\log{R})^c\},\quad \text{if}~~\beta_c>1,
%$$
%and
%$$
%\inf_{|x_0|=R}\|u\|_{L^2(B(x_0, 1))}\gtrsim \exp\{-CR(\log{R})^{c\log\log{R}}\},\quad \text{if}~~\beta_c<1.
%$$

Now we turn to the higher order case.  Weak and strong unique continuation properties for higher order elliptic equations have been studied by many authors, see e.g. \cite{W}, \cite{CG}, \cite{CGT}, \cite{L} and references therein. However, it seems that quantitative results for higher order operators are quite few. In a recent paper by Zhu \cite{Z}, he obtained vanishing order of solutions of polyharmonic equations by using the monotonicity property of a variant of frequency function, where its application to strong unique continuation problems was first observed by Garofalo and Lin \cite{GL}. As a corollary,  it was shown that for $P=(-\Delta)^m$, and if $u$ is a solution to \eqref{1.1} with $n\ge 4m$, then
$$
M(R)\gtrsim \exp\{-CR^{2m}\log{R}\}.
$$
We shall show that the condition $n\ge 4m$ is not necessary and the same bound $\frac43$ is still valid for power of Laplacian. Instead of using frequency function and Sobolev estimates, we improve this bound by noticing that a iteration of the Carleman estimates used in \cite{BK} will allow us to  follow Bourgain and Kenig's approach. Our first result is
\begin{theorem}\label{thm1.1}
Let $P=(-\Delta)^m$, and $u$ satifies \eqref{1.1} and \eqref{1.2}, then
$$
M(R)\gtrsim \exp\{-CR^{\frac43}\log{R}\}.
$$
\end{theorem}
Currently, we don't know yet whether the bound $\frac43$ here is also optimal (up to logarithmic loss) for $(-\Delta)^m$, $m>1$. Nevertheless, in dimension 2, we're able to show that for any $\epsilon>0$, there exists some fourth order elliptic operators, such that  the lower bound can be improved to $\frac87+\epsilon$. Furthermore, we shall prove this bound is essentially sharp (up to $\epsilon$-power loss) by constructing a Meshkov type example.
\begin{theorem}\label{thm1.2}
For any $\epsilon>0$, Let $P=P_1P_2$ in $\mathbb{R}^2$, where $P_1=\Delta_{\mathbb{R}^2}$, $P_2=\partial_{x_1}^2+(1+\frac{\epsilon}{2})\partial_{x_2}^2$. Assume $u$ satisfies \eqref{1.1} and \eqref{1.2}, then
$$
M(R)\gtrsim \exp\{-CR^{\frac87+\epsilon}\log{R}\}.
$$
Furthermore, there exists  nontrivial bounded functions $u, V$ satisfying \eqref{1.1} and
$$
u(x)\lesssim e^{-C|x|^{\frac87}}.
$$
\end{theorem}
\begin{remark}\label{rmk1.3}
Although the operator $P$ above can be view as an "$\epsilon$ perturbation" of $\Delta^2$ in dimension 2, it seems that the order $\frac87$ can not be derived for $\Delta^2$ in this way since we shall see in section \ref{sec3} (see Example \ref{exa3.2}) that no weight function satisfies the strong pseudoconvex condition with respect to $\Delta^2$.
\end{remark}
The paper is organized as follows. In section \ref{sec2}, we prove Theorem \ref{thm1.1}, in addition to the Carleman estimates, we also need an interior regularity lemma to deal with the lower order terms. Section \ref{sec3} is devoted to prove Theorem\ref{thm1.2} by using the method similar to \cite{CGT}, which concerns the pseudo-convex weight functions (with respect to $P$). Throughout the paper, $C$ and $C_j$ denote absolute positive constants whose dependence will be specified whenever necessary. The value of C may vary from line to line.

\section{Proof of Theorem \ref{thm1.1}}\label{sec2}
We start with the following Carleman type inequality
\begin{lemma}\label{lem2.1}
There are constants $C_1, C_2, C_3$ and an increasing function $\omega=\omega(r)$ for $0<r<10$, such that
$$
\frac1C_{1}\leq \frac{\omega(r)}r{}\leq C_1,
$$
and for all $f\in C_{0}^{\infty}(B(0,10)\setminus\{0\})$, $\tau>C_2$, we have
\begin{equation}\label{equ2.1}
\tau^{3m}\int{\omega(|x|)^{-1-2\tau}|f|^2\, dx}\leq C_3\int{\omega(|x|)^{3m-1-2\tau}|\Delta^m f|^2\, dx}.
\end{equation}
\end{lemma}
\begin{proof}
In the case $m=1$, the result is due to  Lemma 3.15 in \cite{BK}, while the general result can be deduced by applying this $m$ times and noting that  $\tau^{3m}\lesssim \prod_{j=0}^{m-1}(\tau-\frac32 j)^3$.
\end{proof}
In order to prove Theorem \ref{thm1.1}, we shall also need the following interior regularity property of elliptic operators, which can be thought of as the $L^{\infty}$ version of Theorem 17.1.3 in \cite{H}.
\begin{lemma}\label{lem2.2}
Assume $P(D)$ is homogeneous and elliptic of order $2m$. Let $X$ be an open set containing $0$, and denote $d(x)$ the distance from $x\in X$ to $\complement{X}$, the complement of $X$. If $P(D)u\in L^{\infty}$ and $u\in L^{\infty}$, then it follows that for $|\alpha|<2m$,
\begin{equation}\label{equ2.2}
\|d(x)^{|\alpha|}D^{\alpha}u\|_{L^{\infty}(X)}\leq C(\|d(x)^{|2m|}P(D)u\|_{L^{\infty}(X)}+\|u\|_{L^{\infty}(X)}).
\end{equation}
\end{lemma}

\begin{proof}
The proof is essentially similar to Theorem 17.1.3 in \cite{H}, we sketch the proof here for the sake of completeness. First, we claim that for any $A>0$, $|\alpha|<2m$,
$\frac{A^{|\alpha|}\xi^{\alpha}}{1+P(A\xi)}$ is a $L^1$ multiplier with bound independent of $A$, hence a $L^{\infty}$ multiplier by duality. In fact, by scaling, it suffices to assume $A=1$, furthermore, we note that for $|\alpha|<2m$,
$$
|D^{\beta}(\frac{\xi^{\alpha}}{1+P(\xi)})|\leq C_\beta (1+|\xi|)^{-1-|\beta|},
$$
thus the claim follows from Bernstein's theorem (see e.g.\cite{G}). So we have the following
$$
\sum_{|\alpha|<2m}{A^{2m-|\alpha|}\|D^{\alpha}u\|_{L^{\infty}}}\leq C(\|P(D)u\|_{L^{\infty}}+A^{2m}\|u\|_{L^{\infty}}), \quad  A>0 .
$$
Then we can proceed as Theorem 17.1.3 in \cite{H} with minor changes. Applying the above estimates to $v=u\cdot\chi(\frac{x-y}{R})$, where
$y\in X$, and $d(y)\ge 2R$ and $\chi\in C_0^{\infty}(B(0,1))$ which is equal to 1 in $B(0,\frac12)$, and expanding $P(D)v$ by Leibniz'
formula with $A=M/R$, gives
\begin{align*}
\sum_{|\alpha<2m|}{M^{2m-|\alpha|}}R^{\alpha}\sup_{B(y, \frac{R}{2})}|D^{\alpha}u|&\leq C(R^{2m}\sup_{B(y, R)}|P(D)u|+\sum_{|\alpha<2m|}R^{\alpha}\sup_{B(y, R)}|D^{\alpha}u|\\
&+M^{2m}\sup_{B(y, R)}|u|),
\end{align*}
where $M$ is some large constant.With $R_0$ to be chosen later, we define
$$
R(y)=\min\{R_0, \frac{d(y)}{2}\}.
$$
Since $|R(x)-R(y)|\leq\frac{|x-y|}{2}$, then with a new constant independent of $R_0$, we have
\begin{align}\label{equ2.22}
\sum_{|\alpha<2m|}{M^{2m-|\alpha|}}&\sup_{B(y, \frac{R(y)}{2})}R(x)^{\alpha}|D^{\alpha}u|\leq C(\sup_{B(y, R(y))}R(x)^{2m}|P(D)u|\nonumber \\
&+\sum_{|\alpha<2m|}\sup_{B(y, R(y))}R(x)^{\alpha}|D^{\alpha}u|+M^{2m}\sup_{B(y, R)}|u|),
\end{align}

Now we take sup norm with respect to $y\in X$, and absorb the second term in the right hand side of \eqref{equ2.22} to the left hand side, which gives \eqref{equ2.2}.
\end{proof}

\begin{proof}[Proof of Theorem \ref{thm1.1}]
Define $u_1(x)=u(ARx+x_0)$ with some small but fixed constant $A$ to be specified later. Since $u$ satisfies \eqref{1.1}, we have
\begin{equation}\label{equ2.3}
|u_1|\leq C, \quad |\Delta^{m}u_1|\leq C(AR)^{2m}|u_1|.
\end{equation}
Assume as we may  $\max_{|x|=\frac1A}|u_1(x)|=1$, First, we have for $|\alpha|<2m$,
\begin{equation}\label{equ2.4}
|D^{\alpha}u_1|\leq C(AR)^{|\alpha|},
\end{equation}
Now choose a bump function $\zeta\in C_0^{\infty}(\frac{1}{4R}<|x|<4)$, with $\zeta=1$, if $\frac{1}{3R}\leq|x|\leq 3$ such that the following estimates
\begin{equation}\label{equ2.5}
\begin{cases}
|\zeta^{\alpha}|\leq CR^{|\alpha|}, \quad \text{if}\quad |x|\leq \frac{1}{3R},\\[4pt]
|\zeta^{\alpha}|\leq C_{\alpha}, \quad \text{if} \quad |x|\ge 3.
\end{cases}
\end{equation}
hold. Applying \eqref{equ2.1} to $f=u_1\zeta$ gives
\begin{equation}\label{equ2.6}
\begin{aligned}
\tau^{3m}\int{\omega^{-1-2\tau}\zeta^2|u_1|^2}&\leq C\int{\omega^{3m-1-2\tau}\zeta^2|\Delta^m u_1|^2}\\
&+\{\int_{\frac{1}{4R}\leq |x|\leq\frac{1}{3R}} +\int_{3\leq |x|\leq 4}\}\omega^{3m-1-2\tau}\sum_{|\alpha|<2m}{|D^{2m-\alpha}\zeta|^2|D^{\alpha}u_1|^2}\\
&\triangleq I_1+I_2.
\end{aligned}
\end{equation}
By \eqref{equ2.3}, we have
$$
I_1\leq C(AR)^{4m}\int{\omega^{-1-2\tau}\xi^2|u_1|^2}.
$$
Choose
$$
\tau\sim R^{\frac43},
$$
we can absorb the term $I_1$ into the left hand side of \eqref{equ2.6}. To deal with the term $I_2$, we note that by \eqref{equ2.4} and \eqref{equ2.5}, one has
$$
\int_{\frac{1}{4R}\leq |x|\leq\frac{1}{3R}}{\omega^{3m-1-2\tau}\sum_{|\alpha|<2m}{|D^{2m-\alpha}\zeta|^2|D^{\alpha}u_1|^2}}\leq R^{2\tau+m+1-2|\alpha|}\max_{|x|\leq \frac{1}{3R}}\sum_{|\alpha|<2m}{|D^{\alpha}u_1|^2},
$$
and
$$
\int_{3\leq |x|\leq 4}{\omega^{3m-1-2\tau}\sum_{|\alpha|<2m}{|D^{2m-\alpha}\zeta|^2|D^{\alpha}u_1|^2}}\leq C(AR)^{4m-2}\omega(3)^{3m-1-2\tau}.
$$
Therefore
\begin{equation}\label{equ2.7}
\begin{aligned}
\frac{\tau^{3m}}{2}\int{\omega^{-1-2\alpha}\zeta^2|u_1|^2}&\leq C\{R^{2\tau+m+1-2|\alpha|}\max_{|x|\leq \frac{1}{3R}}\sum_{|\alpha|<2m}{|D^{\alpha}u_1|^2}\\
&+(AR)^{4m-2}\omega(3)^{3m-1-2\tau}\}
\end{aligned}
\end{equation}
Let now $u_1(a)=1$ for some $a\in \mathbb{R}^n$, $|a|=\frac1A$, thanks to \eqref{equ2.4}, one has
\begin{align}\nonumber
|u_1(x)|\ge \frac12, \quad \text{if}\quad |x-a|\leq \frac{1}{CAR}.
\end{align}
As the same in \cite{BK}, we can choose $A$ such that the last term in \eqref{equ2.7} can also be absorbed to the left hand side of \eqref{equ2.7}. Now apply Lemma  \ref{lem2.2} with $X=B(0,\frac{1}{R})$ and use \eqref{equ2.3} , we obtain
\begin{align}\nonumber
R^{-CR^{\frac43}} &\leq \max_{|x|\leq \frac{1}{3R}}\sum_{|\alpha|<2m}{|D^{\alpha}u_1|^2}\\
&\leq C\sum_{|\alpha|<2m}(R^{|\alpha|-2m}\max_{|x|\leq \frac{1}{R}}|\Delta^m u_1|+R^{|\alpha|}\max_{|x|\leq \frac{1}{R}}|u_1|)\nonumber\\
&\leq CR^{2m-1}\max_{|x|\leq \frac{1}{R}}|u_1|,\nonumber
\end{align}
which proves the theorem.
\end{proof}
\begin{remark}\label{rmk2.3}
In \cite{M}, Meshkov showed that if $u\in H_{2}^{loc}(\Omega_{\rho})$, where $\Omega_{\rho}=\mathbb{R}^n\setminus B(0, \rho)$, and satisfies
$\Delta u-Vu=0$, for some bounded potential $V$, and  if for any $\tau>0$,
\begin{equation}\label{equ2.9}
\int_{\Omega_{\rho}}|u|^2\exp\{2\tau|x|^{\frac43}\}<\infty,
\end{equation}
then $u\equiv 0$. We note that this result can also be generalized to the case $(-\Delta)^m$ by assuming $u\in H_{2m}^{loc}(\Omega_{\rho})$ and the above growth condition \eqref{equ2.9}, since on the one hand, the following Carleman estimates
$$
\tau^{3m}\int{|v|^2r\exp\{2\tau r^{\frac43}\}\, drd\omega}\leq C\int{|\Delta^m v|^2r\exp\{2\tau r^{\frac43}\}\, drd\omega},
$$
can be easily deduced from Lemma 1 in \cite{M}
\footnote[1]{we thank Jiuyi Zhu for  pointing out this to us.} , and on the other hand, the condition \eqref{equ2.9} allows us to obtain a weighted interior $L^2$ regularity estimates in each annulus, since the weight $e^{2\tau|x|^{\frac43}}$ is bounded both from below and above in such annulus, and we can sum over the annulus to get a global one (with a different $\tau$).
\end{remark}

\begin{remark}\label{rmk2.4}
It seems that the example constructed in \cite{M} is not enough to show the sharpness for the power of Laplacian, though the constructions indicate that in dimension 2, there exists a nontrivial solution $u$ of the equation $\Delta^2 u + Vu=0$ with some bounded $V$, such that $|u(x)|\leq C \exp\{-c|x|^{\frac87}\}$, see Section \ref{sec3} below for the case of "perturbations" of $\Delta^2$.

\end{remark}
\section{Proof of Theorem \ref{thm1.2}}\label{sec3}

First, we recall the following notion of pseudo-convex weight fucntions.
\begin{definition}\label{def3.1}
Let $P$ be principally normal in $X\subset\mathbb{R}^n$, with principal symbol $p$.  A $C^2$ function $\varphi$ is called strongly pseudo-convex with respect to $P$ at $x_0$ if
\begin{align*}
\Re\{\bar{p}, \{p, \varphi\}\}(x_0, \xi)>0,\quad \text{whenever} \, p(x_0,\xi)=0, \, \xi\in \mathbb{R}^n\setminus \{0\},
\end{align*}
and
\begin{align*}
\{\bar{p}(x,\xi-i\tau\nabla\varphi), p(x,\xi+i\tau\nabla\varphi)\}/2i\tau >0\quad \text{on} \, \{ p(x,\xi+i\tau\nabla\varphi)=0,\, \tau>0,\\
(\xi, \tau)\neq 0\},
\end{align*}
where $\{p, q \}=\sum{(\frac{\partial p}{\partial \xi_j}\frac{\partial q}{\partial x_j}-\frac{\partial q}{\partial \xi_j}\frac{\partial p}{\partial x_j})}$ is the Poisson bracket of $p$ and $q$.
\end{definition}

In particular, if $P$ is elliptic, then $\varphi$ is strongly pseudo-convex with respect to $P$ if
\begin{equation}\label{equ3.1}
\{\Re{p(x,\xi+i\tau\nabla\varphi)}, \Im{p(x,\xi+i\tau\nabla\varphi)}\}>0 \quad \text{on}~~ p(x,\xi+i\tau\nabla\varphi)=0.
\end{equation}

\begin{example}\label{exa3.2}
 (i) Consider  $P=-\Delta$, it's easy to see that $\varphi$ is strongly pseudo-convex with respect to $-\Delta$ if and only if
 $$
 (\xi, H(\varphi)\xi)+\tau^2(\nabla\varphi, H(\varphi)\nabla\varphi)>0
 $$
 on the set defined by
 \begin{align}\label{equ3.2}
 \begin{cases}
 |\xi|^2=\tau^2|\nabla\varphi|^2\\[4pt]
  \xi\cdot\nabla\varphi=0
  \end{cases}
 \end{align}
  where $(\cdot,\cdot)$ is the standard inner product in Euclidian space, and $H(\varphi)$ is the Hessian of $\varphi$. Let $\varphi_1=-\ln{|x|}-\int_{0}^{|x|}{\frac{e^{-t}-1}{t} dt}$ and assume  $0\notin X$, which is the weight (singular at the origin) used in Section \ref{sec2} (see \cite{BK}) , in this case
 $$
 H(\varphi_1)=-\frac{e^{-|x|}}{|x|^2}Id+ (\frac{e^{-|x|}}{|x|^3}+\frac{2e^{-|x|}}{|x|^4})x\cdot x^t,
 $$
 where Id is the identity matrix, thus on the set defined by \eqref{equ3.2}, one has
 \begin{align*}
 (\xi, H(\varphi_1)\xi)+\tau^2(\nabla\varphi_1, H(\varphi_1)\nabla\varphi_1)=\tau^2\frac{e^{-3|x|}}{|x|^3}>0,
 \end{align*}
 which implies that $\varphi_1$ is strongly pseudo-convex with respect to $-\Delta$ in $X$. We note also that other
 strongly (singular) pseudo-convex weight functions include $\varphi_2(x)=(\ln{|x|})^2$, $\varphi_3(x)=-\ln(|x|+\lambda|x|^2)$, where $\lambda>1$. These weight functions are very useful in obtaining strong unique continuation theorems for second order elliptic operators with principal part $\Delta$, see e.g. \cite{H83}, \cite{R}, \cite{L}.

(ii) $P=(-\Delta)^m, m>1$. This is quite different from the case $m=1$. In this case, no functions satisfy the convex condition \eqref{equ3.1}. In fact, denote $p_m=|\xi|^{2m}$, if such a function exists, then
\begin{align*}
&\{\bar{p}_m(\xi+i\tau\nabla\varphi), \bar{p}_m(\xi+i\tau\nabla\varphi)\}/2i\tau\\
 &=\{(|\xi|^2-\tau^2|\nabla\varphi|^2-2i\tau\xi\cdot\nabla\varphi)^m, (|\xi|^2-\tau^2|\nabla\varphi|^2+2i\tau\xi\cdot\nabla\varphi)^m\}\\
&=m^2[(|\xi|^2-\tau^2|\nabla\varphi|^2)^2+4\tau^2(\xi\cdot\nabla\varphi)^2]^{m-1}\{\bar{p}_1(\xi-i\tau\nabla\varphi), p_1(\xi+i\tau\nabla\varphi)\}/2i\tau\\
&\equiv 0
\end{align*}
 on $\{p_m(\xi+i\tau\nabla\varphi)=0\}$, i.e., the set defined by \eqref{equ3.2}.

\end{example}
Now we are in a position to prove Theorem \ref{thm1.2}, the key point is the following Carleman estiamtes.
\begin{lemma}\label{lemma3.3}
Let $\varphi(r)=r^{-\alpha}$, $r=|x|$, $P=P_1P_2$, where $P_1=\Delta_{\mathbb{R}^2}$, $P_2=\partial_{x_1}^2+b\partial_{x_2}^2$, $b>0, b\neq 1$. Further
suppose
$$
\alpha>\max\{\frac{1}{b}-1, b-1\}.
$$
Then we have
\begin{equation}\label{equ3.3}
\tau^{-1}\|(r|\nabla\varphi|)^{-\frac12}e^{\tau\varphi(r)}u\|_{4, \tau}^2\leq C\|e^{\tau\varphi(r)}Pu\|_{L^2}^2, \quad u\in C_0^{\infty}(B(0, 10)\setminus \{0\}),
\end{equation}
where $\|v\|_{4, \tau}^2=\|\partial^4v\|_{L^2}+\||\tau\nabla\varphi|^4v\|_{L^2}^2$, and C is some positive constant does not depend on $\tau$.
\end{lemma}
\begin{proof}
We first note that it suffices to establish \eqref{equ3.3} for $u\in C_0^{\infty}(\frac12<|x|<1)$, since if this is true, then the same scaling arguments as in \cite{CGT} will imply \eqref{equ3.3}. To this end, we shall prove that $\varphi$ satisfies the following form of strongly pseudo-convex condition
\begin{equation}\label{equ3.4}
\{\Re{p_{\tau}}, \Im{p_{\tau}}\}\geq \frac{C}{r}(|\xi|+\tau|\nabla\varphi|)^7 \quad \text{on} ~~ p_{\tau}=0,
\end{equation}
where $p(\xi)=(\xi_1^2+\xi_2^2)(\xi_1^2+b\xi_2^2)$, and $p_{\tau}=p(\xi+i\tau\nabla\varphi)$. In fact, we notice that the following identity holds
\begin{align*}
\{\Re{p_{\tau}}, \Im{p_{\tau}}\}=\{\Re{p_{1, \tau}}, \Im{p_{1, \tau}}\}|p_{2, \tau}|^2+\{\Re{p_{2, \tau}}, \Im{p_{2, \tau}}\}|p_{1, \tau}|^2, \quad \text{on } ~~p_{\tau}=0.
\end{align*}
 The condition $b\neq 1$ indicates  that $Char P_{1, \tau}\bigcap Char P_{2, \tau}=\emptyset$. We note that by homogeneity consideration one can let $\tau=1$. Without loss of generality, we can assume $p_{2, \tau}=0$, that is
\begin{equation}\label{equ3.5}
\begin{gathered}
\xi_1^2+b\xi_2^2=(\partial_{x_1}\varphi)^2+b(\partial_{x_2}\varphi)^2,\\
\xi_1\partial_{x_1}\varphi+b\xi_2\partial_{x_2}\varphi=0,
\end{gathered}
\end{equation}
which implies that
$$
\xi_1^2+(\partial_{x_1}\varphi)^2=b(\xi_2^2+(\partial_{x_2}\varphi)^2),
$$
so
\begin{align}\label{equ3.6}
|p_{1, \tau}|^{2} &=(|\xi|^2-|\nabla\varphi|^2)^2+4(\xi\cdot \nabla\varphi)^2\nonumber \\
&=(b-1)^2(\xi_2^2+(\partial_{x_2}\varphi)^2)^2 \nonumber\\
&\geq C_b(|\xi|+|\nabla\varphi|)^4.
\end{align}
On the other hand, if we denote by
$$
\xi_b=(\xi_1,~ b\cdot\xi_2)^{T}, \quad \nabla\varphi_b=(\partial_{x_1}\varphi,~ b\cdot\partial_{x_2}\varphi)^{T},
$$
one has
\begin{align*}
\{\Re{p_{2, \tau}}, \Im{p_{2, \tau}}\}&=-2\alpha|x|^{-\alpha-2}(|\xi_b|^2+|\nabla\varphi_b|^2)\\
&+2\alpha(\alpha+2)|x|^{-\alpha-4}[(x\cdot\xi_b)^2+(x\cdot\nabla\varphi_b)^2]
\end{align*}
Thanks to \eqref{equ3.5}, we have $x\cdot\xi_b=0$ and $|\xi_b|^2=b\alpha^2|x|^{-2\alpha-2}$, thus, we obtain
\begin{align*}
\{\Re{p_{2, \tau}}, \Im{p_{2, \tau}}\}&=2\alpha^3|x|^{-3\alpha-4}[(\alpha+2)(\frac{x_1^2+bx_2^2}{|x|^2})^2-b-\frac{x_1^2+b^2x_2^2}{|x|^2}],
\end{align*}
it follows from our assumption on $\alpha$ that there exists a positive constant c (depending on $b$) such that
$$
(\alpha+2)(\frac{x_1^2+bx_2^2}{|x|^2})^2-b-\frac{x_1^2+b^2x_2^2}{|x|^2}>c>0.
$$
Note also that on $p_{2, \tau}=0$, one has  the  relation $|\xi|\sim |\nabla\varphi|$, hence we have
\begin{equation}\label{equ3.65}
\{\Re{p_{2, \tau}}, \Im{p_{2, \tau}}\}\geq C\frac{(|\xi|+|\nabla\varphi|)^3 }{|x|}.
\end{equation}
Combine \eqref{equ3.6} and \eqref{equ3.65}, we obtain \eqref{equ3.4}, and the desired Carleman estimates follows by standard arguments (see e.g. \cite{CGT}, \cite{H}).
\end{proof}

\begin{proof}[Proof of Theorem \ref{thm1.2}]
For any $0<\epsilon<1$, set $b=1+\frac{\epsilon}{2}$, $\varphi(r)=r^{-\epsilon}$. Applying Lemma \ref{lemma3.3} to $P=(\partial_{x_1}^2+b\partial_{x_2}^2)\Delta$, one has
$$
\tau^7\int{|\nabla\varphi|^7|x|^{-1}e^{2\tau\varphi}|u|^2}\leq C\int{e^{2\tau\varphi}|u|^2}, \quad u\in C_0^{\infty}(B(0, 10)\setminus \{0\}),
$$
which plays the same role as Lemma \ref{lem2.1} in the proof of Theorem \ref{thm1.1}. Following the same way in the previous section we prove that
$$
M(R)\gtrsim \exp\{-CR^{\frac87+\epsilon}\log{R}\},
$$
which finishes the first part of Theorem \ref{thm1.2}.

We end the proof by constructing  a Meshkov type example to show that the bound $\frac87$ is optimal. For simplicity, we shall assume
$P=\Delta_{\mathbb{R}^2}(\partial_{x_1}^2+2\partial_{x_2}^2)$. The key point is the following observation
\begin{proposition}\label{pro3.4}
Suppose $\rho>0$ is large enough, choose $n,k\in \mathbb{Z}$ such that $|n-\rho|\leq 1$, $|k-8\rho^{\frac87}|\leq 6$. Then in the annulus $\rho\leq |x|\leq 7\rho^{\frac37}$, there exists a solution u satisfying \eqref{1.1} and \eqref{1.2} such that the following properties hold\\
(i) If $r\in [\rho, \rho+0.1\rho^{\frac37}]$, then $u(r,\varphi)=r^{-n}e^{-in\varphi}$. If $r\in [\rho+6.9\rho^{\frac37}, \rho+7\rho^{\frac37}]$, then $u(r,\varphi)=ar^{-n-k}e^{-i(n+k)\varphi}$, where $a$ is some positive constant.\\
(ii)Let $m(r)=\max\{|u(r, \varphi)|, 0\leq \varphi\leq 2\pi\}$, there exists an absolute constant C, which does not depend on $\rho, n, k$ such that for
$r\in [\rho, \rho+7\rho^{\frac37}]$,
\begin{equation}\label{equ3.7}
\ln{m(r)}-\ln{m(\rho)}\leq -C\int_{\rho}^{r}{t^{\frac17}\, dt}+C.
\end{equation}
\end{proposition}

\begin{proof}
We remark here that the constants $C, C_j$ appeared in the proof below can all be chosen independent of $\rho, n, k$.

Step 1 (when $\rho\leq |x|\leq 2\rho^{\frac37}$) The solution $u_1=r^{-n}e^{-in\varphi}$ is rearranged to $u_2=-br^{-n+2k}e^{-iF(\varphi)}$,
 where $b=(\rho+\rho^{\frac37})^{-2k}$, $F(\varphi)=(n+2k)\varphi+\Phi(\varphi)$. Moreover, $\Phi(\varphi)$ satisfies
 \begin{align}
 |\Phi^{(j)}(\varphi)|\leq C\rho^{\frac87j-\frac47}, \quad j=0,1,2,\ldots.\label{equ3.8}\\
 \Phi(\varphi)=-4k\varphi+b_m, \quad \text{in} \quad |\varphi-\varphi_m|\leq \frac T5,\label{equ3.9}
 \end{align}
  where $b_m\in \mathbb{R}$, $T=\frac{\pi}{n+k}$, $\varphi_m=mT$, $m=0,1,\ldots, 2(n+k)-1$. For the existence of such $\Phi$, see \cite{M} for the detail. Next there exists  $C^{\infty}$ functions $\psi_1, \psi_2$ with $\psi_1(r)=1$, if $r\leq \rho+\frac{13}{7}\rho^{\frac37}$, $\psi_1(r)=0$, if $r\ge \rho+1.9\rho^{\frac37}$, and $\psi_2(r)=1$, if $r\ge \rho+\frac{1}{7}\rho^{\frac37}$, $\psi_2(r)=0$, if $r\leq \rho+0.1\rho^{\frac37}$ such that
  \begin{equation}\label{equ3.10}
  \psi_{k}^{j}(r)\leq C \rho^{-\frac37 j}, \quad k=1,2, \quad j=0, 1, 2,\ldots
  \end{equation}
  Now set $u=\psi_1u_1+\psi_2u_2$, by \eqref{equ3.9}, we have $\Delta u=0$ in the region
  $$
  \{\rho+\frac{1}{7}\rho^{\frac37}\leq r\leq \rho+\frac{13}{7}\rho^{\frac37}, |\varphi-\varphi_m|\leq \frac T5\},\quad  m=0, 1,\ldots 2(n+k)-1.
  $$
  Since $|\frac{u_2}{u_1}|=(\frac{r}{\rho+\rho^{\frac37}})^{2k}$, it follows that there exists some positive constant C (say 10), such that
  \begin{align}
  |\frac{u_2}{u_1}|&\leq e^{-C}, \quad r\in[\rho, \rho+\frac{1}{7}\rho^{\frac37}], \label{equ3.11}\\
  |\frac{u_2}{u_1}|&\ge e^{-C}, \quad r\in[\rho+\frac{13}{7}\rho^{\frac37}, \rho+2\rho^{\frac37}]. \label{equ3.12}
  \end{align}
  First we consider the annulus $\rho\leq |x|\leq \frac17 \rho^{\frac37}$, then it's easy to see from \eqref{equ3.11} that
  \begin{equation}\label{equ3.13}
   |u|\ge e^C|u_2| \quad \text{for some} \quad C>0.
  \end{equation}

 Now we estimate $(\partial_{x_1}^2+2\partial_{x_2}^2)\Delta u$ in this region. Note that in the polar coordinates, one has
 \begin{align*}
 \Delta_{\mathbb{R}^2}&=\frac{\partial^2}{\partial r^2}+\frac1r\frac{\partial}{\partial r}+\frac{1}{r^2}\frac{\partial^2}{\partial \varphi^2},
 \end{align*}
 \begin{align*}
 \partial_{x_1}^2+2\partial_{x_2}^2 &=(1+\sin^2\varphi)\frac{\partial^2}{\partial r^2}+(\frac{1+\cos^2\varphi}{r}+\frac{\sin{2\varphi}}{r}\frac{\partial}{\partial \varphi})\frac{\partial}{\partial r}\\&-\frac{\sin{2\varphi}}{2r^2}\frac{\partial}{\partial \varphi}+\frac{1+\cos^2\varphi}{r^2}\frac{\partial^2}{\partial \varphi^2}. \nonumber
  \end{align*}
 Then we write
  $$\Delta u=\Delta(\psi_2u_2)=\psi_2\Delta u_2+2\frac{\partial \psi_2}{\partial r}\frac{\partial u_2}{\partial r}+u_2\Delta\psi_2,$$
  and
  \begin{equation}
  \Delta u_2=\frac{g_1(\varphi)}{r^2}\cdot u_2, \nonumber
  \end{equation}
  where $g_1(\varphi)=(n-2k)^2-(n+2k+\Phi^{'})^2+i\Phi^{''}$, hence, it follows from \eqref{equ3.8} that
  $$
  |\frac{ \Delta u_2}{u_2}|\leq \frac{nk}{r^2}\leq C\rho^{-\frac27}.
  $$
  A further computation shows that
    \begin{align*}
    |\frac{\partial^2}{\partial r^2}(\frac{g_1(\varphi)}{r^2} u_2)|&=|\frac{g_1(\varphi)(n-2k)(n-2k+1)}{r^4}u_2|\\
    &\leq C\frac{n^3k}{r^4}|u_2| \leq C|u_2|,
  \end{align*}

  \begin{align*}
    |\frac{\partial^2}{\partial r^2}(\frac{\partial \psi_2}{\partial r} \frac{\partial u_2}{\partial r})|&=|\frac{\partial^3 \psi_2}{\partial r^3} \frac{\partial u_2}{\partial r}+2\frac{\partial^2 \psi_2}{\partial r^2} \frac{\partial^2 u_2}{\partial r^2}+\frac{\partial \psi_2}{\partial r} \frac{\partial^3 u_2}{\partial r^3}|\\
    &\leq C(\frac{n}{r}\cdot\rho^{-\frac97}+\frac{n^2}{r^2}\cdot\rho^{-\frac67}+\frac{n^3}{r^3}\cdot\rho^{-\frac37})|u_2|\\
    & \leq C|u_2|,
  \end{align*}

  \begin{align*}
   |\frac 1r\frac{\partial}{\partial \varphi}\frac{\partial}{\partial r}(\frac{g_1(\varphi)}{r^2} u_2))|&=|\frac{-2g_1'(\varphi)}{r^4} u_2+\frac{g_1(\varphi)}{r^4} \frac{\partial u_2}{\partial r}+\frac{g_1'(\varphi)}{r^3} \frac{\partial u_2}{\partial r}+\frac{g_1(\varphi)}{r^3} \frac{\partial^2 u_2}{\partial r\partial \varphi}|\\
   &\leq C\rho^{-\frac37}\frac{n^3}{r^3}|u_2|\leq C|u_2|,
  \end{align*}

 \begin{align*}
   |\frac 1r\frac{\partial}{\partial \varphi}\frac{\partial}{\partial r}(\frac{\partial \psi_2}{\partial r}\frac{\partial u_2}{\partial r})|&\leq
    C\rho^{-\frac37}\frac{n^3}{r^3}|u_2|\leq C|u_2|,
  \end{align*}

  and
  \begin{align*}
    |\frac{1}{r^2}\frac{\partial^2}{\partial \varphi^2}(\frac{g_1(\varphi)}{r^2} u_2)|& =|-\frac{((F')^2+iF^{''})g_1-iF'g_1'+g_1^{''}}{r^4} u_2| \\
    &\leq C\frac{n^3 k+n\Phi^{'''}+\Phi^{''}\Phi^{'}+n^3\rho^{\frac{20}{7}+\rho^4}}{r^4}|u_2|\\
    &\leq C|u_2|,
   \end{align*}
   where we have used the fact that $|g_1'(\varphi)|\leq C \rho^{\frac{20}{7}}$, and $|g_1''(\varphi)|\leq C \rho^4$.
  After a direct computation of other terms, we obtain the following
  \begin{equation}\label{equ3.14}
  |\frac{(\partial_{x_1}^2+2\partial_{x_2}^2)\Delta u_2}{u_2}|\leq C, \quad \rho\leq |x|\leq 2\rho^{\frac37}.
  \end{equation}
  Combining \eqref{equ3.13} and \eqref{equ3.14}  yields
  \begin{equation}\label{equ3.15}
  |\frac{(\partial_{x_1}^2+2\partial_{x_2}^2)\Delta u}{u}|\leq C,\quad \rho\leq |x|\leq \frac17 \rho^{\frac37}.
  \end{equation}
  Similarly, by \eqref{equ3.12} and arguments above, it follows that \eqref{equ3.15} is valid when $\rho+\frac{13}{7}\rho^{\frac37}\leq |x|\leq 2\frac17 \rho^{\frac37}$. In the remaining annulus sectors
  \begin{align*}
  P_m&=\{(r, \varphi),\rho+\frac{1}{7}\rho^{\frac37}\leq r\leq \rho+\frac{13}{7}\rho^{\frac37}, \varphi_m+\frac T5\leq \varphi\leq \varphi_m+\frac {4T}5\}, \\
  m&=1, 2,\ldots 2(n+k)-1.
 \end{align*}
 One argues the same as in \cite{M} and in this region, we have
 $$
 |u|\ge C|u_2|
 $$
 Using the fact that $u_1$ is harmonic in $P_m$ and $u_2$ satisfies \eqref{equ3.14}, we conclude that \eqref{equ3.15} is also valid in each $P_m$.

 Step 2 (when $\rho+2\rho^{\frac37}\leq |x|\leq \rho+3\rho^{\frac37}$).  The solution $u_2=-br^{-n+2k}e^{-iF(\varphi)}$ is rearranged to $u_3=-br^{-n+2k}e^{i(n+2k)\varphi}$. Let $\psi\in C^{\infty}$, and $\psi(r)=1$, if $r\leq \rho+\frac{15}{7}\rho^{\frac37}$, $\psi(r)=0$, if $r\ge \rho+\frac{20}{7}\rho^{\frac37}$, such that
 $$
  \psi^{j}(r)\leq C_j, \quad r\in (0, \infty) \quad j=0, 1, 2,\ldots
 $$
 Now set
 $$u=-br^{-n+2k}\exp[i(\psi\Phi(\varphi)+(n+2k)\varphi)],$$
 we have
 $$
 \Delta u =g_2(r, \varphi)u,
 $$
 where
 \begin{align*}
 g_2&=\frac{(-8nk-2(n+k)\psi\Phi'+(\psi\Phi')^2+i\psi\Phi'')}{r^2}\\
 &+i(-2n+4k+1)\frac{\psi'\Phi}{r}+i\Phi(\frac{\psi'}{r}+\psi'')-(\psi'\Phi)^2
 \end{align*}
 it then follows that
 $$
 |g_2(r,\varphi)|\leq C\frac{nk}{r^2}\leq C\rho^{-\frac27}.
 $$
 To estimate $\frac{(\partial_{x_1}^2+2\partial_{x_2}^2)(g_2u)}{u}$ in this region we note that
   \begin{align*}
    |\frac{\partial^2 }{\partial r^2}(g_2 u)| \leq C\rho^{-\frac27}\frac{n^2}{r^2} |u|\leq C|u|,
   \end{align*}

  \begin{align*}
   |\frac 1r \frac{\partial}{\partial \varphi}\frac{\partial}{\partial r}(g_2 u)|&=|\frac{1}{r}(\frac{\partial^2 g_2}{\partial \varphi\partial r}u+g_2\frac{\partial^2 u}{\partial \varphi\partial r}+\frac{\partial g_2}{\partial \varphi}\frac{\partial u}{\partial r}+\frac{\partial g_2}{\partial r}\frac{\partial u}{\partial\varphi})|\\
    &\leq C(\frac{n^2}{r^2}\rho^{-\frac27}+\frac{n}{r^2}\rho^{-\frac67}) |u|\leq C|u|,
  \end{align*}
  and
   \begin{align*}
    |\frac{1}{r^2}\frac{\partial^2}{\partial \varphi^2}(g_2 u)|&=|\frac{1}{r^2}(\frac{\partial^2 g_2}{\partial \varphi^2}u+\frac{\partial^2 u}{\partial \varphi^2}g_2+2\frac{\partial g_2}{\partial \varphi}\frac{\partial u}{\partial \varphi})| \\
    &\leq C(\frac{\rho^2}{r^2}+\frac{n^2}{r^2}\rho^{-\frac27}+\frac{n}{r^2}\rho^{\frac67})|u|\\
    &\leq C|u|,
   \end{align*}
   where we have used the fact that $|\frac{\partial^jg_2}{\partial r^j}|\leq C \rho^{-\frac27}$, $|\frac{\partial^jg_2}{\partial \varphi^j}|\leq C \rho^{-\frac27+\frac87 j}$, when $j=0, 1, 2$.
   Another direct computation shows that other terms are also controlled by $C|u|$, which implies that \eqref{equ3.15} is valid in $\rho+2\rho^{\frac37}\leq |x|\leq \rho+3\rho^{\frac37}$.

   Step 3 (when $\rho+3\rho^{\frac37}\leq |x|\leq \rho+4\rho^{\frac37}$).  The solution $u_3=-br^{-n+2k}e^{i(n+2k)\varphi}$ is rearranged to $u_4=-b_1r^{-n-2k}e^{i(n+2k)\varphi}$.  First choose $\psi_3\in C^{\infty}$, and $\psi_3(r)=1$, if $r\leq \rho+\frac{22}{7}\rho^{\frac37}$, $\psi_3(r)=0$, if $r\ge \rho+\frac{27}{7}\rho^{\frac37}$, such that $\psi_3^{(j)}(r)\leq C_j$,  if $r\in (0, \infty), j=0, 1, 2,\ldots$. Let
   $g_3(r)=(\frac{\rho+3\rho^{\frac37}}{r})^{4k}$, then it follows that in $\rho+3\rho^{\frac37}\leq |x|\leq \rho+4\rho^{\frac37}$,
   $$
   |g_3^{(j)}(r)|\leq C\frac{k}{r}^j,\quad j=0, 1, 2,\ldots.
   $$
   Next, we set $h(r)=\psi_3+(1-\psi_3)g_3(r)$, and define $u=u_3\cdot h$, so we have
   $$
   \Delta u=h \Delta u_3+2\frac{\partial h}{\partial r}\frac{\partial u_3}{\partial r}+u_3\Delta h
   $$
   and
   $$
   |\frac{\Delta u_3}{u_3}|=\frac{8nk}{r^2}\leq C \rho^{-\frac27}.
   $$
   The remaining computation is similar to Step 2, so we omit it.

   Step 4 ((when $\rho+4\rho^{\frac37}\leq |x|\leq \rho+7\rho^{\frac37}$).) The solution $u_4=-b_1r^{-n-2k}e^{i(n+2k)\varphi}$ is rearranged to $u_5=-a_1r^{-n-2k}e^{i(n+k)\varphi}$ for some constant $a_1$. This step is similar to Step 1. Chose $\psi_4, \psi_5$ with $\psi_4(r)=1$, if $r\leq \rho+6\frac{6}{7}\rho^{\frac37}$, $\psi_4(r)=0$, if $r\ge \rho+6.9\rho^{\frac37}$, and $\psi_5(r)=1$, if $r\ge \rho+4\frac{1}{7}\rho^{\frac37}$, $\psi_5(r)=0$, if $r\leq \rho+4.1\rho^{\frac37}$, we also require that $\psi_4, \psi_5$ satisfy the condition \eqref{equ3.10}. Now define $u=\psi_4u_4+\psi_5u_5$, and as in Step 1, it's not hard to see that \eqref{equ3.15} is valid in this region.
   Now proceed as the same in \cite{M}, we see the solution $u$ satisfies the (ii).
\end{proof}

\begin{proof}[Proof of Theorem \ref{thm1.2}](continuous). In order to use the Proposition \ref{pro3.4} inductively, we note that if we choose $\rho_1$ to be sufficiently large, and define $\rho_{j+1}=\rho_j+7\rho_j^{\frac37}$, and let $n_j=[\rho_j]$, and $k_j=n_{j+1}-n_j$, we have
\begin{align*}
k_j&\leq \rho_{j+1}^{\frac87}-\rho_{j}^{\frac87}+2\leq \rho_{j}^{\frac87}(1+7\rho_{j}^{-\frac47})-\rho_{j}^{\frac87}+2\\
&\leq 8\rho_{j}^{\frac47}+6.
\end{align*}
Therefore, one can choose $k_j$ such that $|k_j-8\rho^{\frac87}|\leq 6$. By using \ref{pro3.4} repeatedly, we have

\begin{align*}
\ln{m(r)}&\leq C-C\int_{\rho_1}^{r}{t^{\frac17}}\, dt+\ln{m(\rho_1)}\\
&\leq C-Cr^{-\frac87}
\end{align*}
which completes the the proof.
\end{proof}
\end{proof}
%
%Finally we end this note by posing the following question:

\section*{Acknowledgements}
The first author is very grateful to Jiuyi Zhu for many helpful discussions
and encouragement during his visit at Johns Hopkins University.


\begin{thebibliography}{00}
\bibitem{BK}  J. Bourgain and C. Kenig, On localization in the Anderson-Bernoulli model in higher
dimensions, Invent. Math. 161 (2005) 389-426.

\bibitem{CG}  Colombini, F. Grammatico, C., Some remarks on strong unique continuation for the
Laplace operator and its powers, Comm. Partial Differential Equations 24 (1999) 1079-1094.

\bibitem{CGT} Ferruccio Colombini, Cataldo Grammatico, Daniel Tataru, Strong uniqueness for second
order elliptic operators with Gevrey coefficients, Math. Res. Lett. 13 (2006) no. 1 15-27.

\bibitem{D}  Davey, B., Some quantitative unique continuation results for
eigenfunctions of the magnetic Schr\"{o}dinger operator. Comm. Partial Differential Equations, 39 (2014) 876-945

\bibitem{G}  L. Grafakos, Classical and Modern Fourier Analysis, Prentice Hall 2004.

\bibitem{GL} Nicola Garofalo and Fang Hua Lin. Monotonicity properties of variational integrals, $A_p$
 weights and unique continuation, Indiana Univ. Math. J. 35(1986) 245-268.

\bibitem{H83} L. H\"{o}rmander, Uniqueness theorems for second order elliptic differential equations, Comm.
Partial Differential Equations 8 (1983) 21-64.

\bibitem{H} L. H\"{o}rmander, The Analysis of Linear Partial Differential Operators, Vols.
3 and 4, Springer-Verlag, 1985.

\bibitem{KL} V.A. Kondratiev, E.M. Landis, Quantitative theory of linear partial differential equations of second
order, in: Partial Differential Equations III, in: Encyclopaedia Math. Sci., vol. 32, Springer-Verlag, Berlin, 1988.

\bibitem{L} C.L. Lin,  Strong unique continuation for m-th power of a Laplacian operator with singular coefficients,
Proc. of AMS 135 (2007) 569-578.

\bibitem{LW} C.L. Lin, Jenn-Nan Wang, Quantitave uniqueness estimates for the general second order elliptic equations, J. of
Funct. Anal. Vol 266 (2014) 5108-5125.

\bibitem{M} V. Z. Meshkov, On the possible rate of decay at infinity of solutions of second order partial
differential equations, Math. USSR Sbornik 72 (1992) 343-360.

\bibitem{R}  Regbaoui, R, Strong uniqueness for second order differential operators, J. Differential Equations 141 (1997) 201-217.

\bibitem{W}   Wensheng Wang, Carleman inequalities and unique continuation for higher-order elliptic differential
 operators. Duke Math. J. 74 (1994) 107-128.

\bibitem{Z} Jiuyi Zhu, Quantitative uniqueness of elliptic equations, to appear in American Journal of Mathematics.


%\bibitem{A} Avramidi, I.;  Schiming, R.;
%\emph{A new explicit expression for the Korteweg-De Vries hierarchy},
% Math. Nachr. \textbf{219} (2000), 45-64.


\end{thebibliography}
\end{document}